\documentclass[fleqn]{amsart}
\usepackage{amsmath}
\usepackage{amssymb, amstext, amsthm, amsfonts,euscript,amscd,mathrsfs}
\usepackage[colorlinks=false]{hyperref}
\usepackage[left=3.2cm,top=3cm,right=3.3cm,bottom=1.4in,asymmetric]{geometry}
\DeclareMathOperator{\Real}{\mathbb{R}}

\DeclareMathOperator{\sign}{sign}
\DeclareMathOperator{\esssup}{ess\sup}

\newcommand{\N}{\mathbb{N}}
\newcommand{\Schr}{Schr\"{o}dinger}
\newcommand{\jap}[1]{\!\left<#1\right>}

\newcommand{\dx}{\partial_x}
\newcommand{\dy}{\partial_y}
\newcommand{\dxi}{\partial_{\xi}}

\newcommand{\norm}[1]{\lVert#1\rVert}
\newcommand{\abs}[1]{\lvert#1\rvert}

\newcommand{\pd}{\partial}
\newcommand{\laplace}{\triangle}
\newcommand{\der}[2]{\frac{\partial #1}{\partial #2}}

\newcommand{\dt}{\pd_t}
 
\newcommand{\tld}[1]{\tilde{#1}}

\newcommand{\R}{\Real}

\newcommand{\eps}{\varepsilon}

\newcommand{\B}{\mathcal{B}}

  \DeclareMathOperator{\supp}{supp}
   
\newcounter{theorem_counter}
\newtheorem{thm}{Theorem}
\newtheorem{lem}[thm]{Lemma}
\newtheorem{prop}[thm]{Proposition}

\newtheorem{coro}[thm]{Corollary}

\newtheorem{defn}[thm]{Definition}

\theoremstyle{remark}
\newtheorem{remark}[thm]{Remark}

\title{A sharp condition for the well-posedness of the linear KdV-type equation}
\author{Timur Akhunov}
\address{Department of Mathematics, University of Calgary,
2500 University Drive NW,  Calgary, Alberta, T2N 1N4, Canada}
\curraddr{}
\email{takhunov@ucalgary.ca}
\date{}

\begin{document}
\maketitle


\begin{abstract}
An initial value problem for a very general linear equation of KdV-type is considered. Assuming non-degeneracy of the third derivative coefficient this problem is shown to be well-posed under a certain simple condition, which is an adaptation of Mizohata-type condition from the {\Schr} equation to the context of KdV. When this condition is violated ill-posedness is shown by an explicit construction. These results justify formal heuristics associated with dispersive problems and have applications to non-linear problems of KdV-type.
\end{abstract}

\maketitle

\section{Introduction}
This paper is concerned with the study of the equation
\begin{equation}
  \label{eq}
  \begin{cases}
    \dt u + Lu =f \text{ for } (t,x) \in (0,T]\times \R\\
    u(0,x)= u_0(x)
  \end{cases}\\ ,\text{ where } L= \sum_{j=0}^3 a_j(t,x)\dx^j
\end{equation}
where $a_j$ are real-valued functions.\\
This is the most general linear form of the KdV, one of the most studied dispersive equations, and used as an important model in understanding behavior of linear and non-linear waves. Such an equation with non-constant dispersive coefficient $a_3$ describes nonisotropic dispersion and its study is of use for the quasi-linear analogues of \eqref{eq}.\\

Another motivation, for the study of the well-posedness of \eqref{eq} is understanding the relative strength of dispersive and non-dispersive effects present in the equation. In particular, from the geometrical optics expansion for the equation, c.f. the classical book of Whitham \cite{Whit}, the dispersive coefficient $a_3$ guides the propagation of the wave packets, while the term $a_2\dx^2$ can lead to the growth of the amplitudes of the wave packets of \eqref{eq}. In light of these heuristics, it is natural to expect that well-posedness requires \textit{non-degeneracy} of $a_3$, which prevents the collapse of the wave packets, namely $0<\eps\le \abs{a_3} \le \frac{1}{\eps}$ for some $\eps$, and a condition on $a_2$ to ensure dispersion dominates anti-diffusion effects. Craig-Goodman \cite{CrGdm90} proved well-posedness in the Sobolev spaces $H^s$ for $a_2\equiv a_1\equiv 0$ under the \textit{non-degeneracy} of coefficient $a_3$ and ill-posedness for some degenerate cases of $a_3$. In a follow up paper, Craig-Kappeler-Strauss \cite{CKS92} proved well-posedness with non-degenerate dispersion and $-a_2\ge 0$, as well as extensions to the quasi-linear analogues. These results were extended in \cite{Akh13} to allow for the "anti-diffusion" in $a_2$, as long as $\jap{x}^{\frac{1}{2}^+}\abs{a_2} \le C$, under some additional assumptions on other coefficients, and to systems of equations. \\

In the current paper, the condition on the diffusion coefficient $a_2$ is extended to a sharp one for the well-posedness in $H^s$, where well-posedness means \textit{existence} of $C^0_{[0,T]}H^s$ distributional solutions of \eqref{eq}, that are \textit{unique} and \textit{depend continuously on data} in the $C^0_{[0,T]}H^s$ topology. Namely a condition on the diffusion coefficient $a_2$ along the flow is obtained, that separates well-posedness from ill-posedness (in the sense of violating continuous dependence) of \eqref{eq} with non-degenerate dispersion. This is qualitatively similar to the necessity of a Mizohato condition $\abs{\sup_{x,t\abs{\omega}=1}\int_0^t\Im b(x+s\omega)\cdot \omega ds} <\infty$ for the well-posedness {\Schr} equation $\dt u + i\laplace u + b(x) \nabla u =0$ in \cite{Miz85}, see also \cite{CKSt95}, \cite{Doi94}, \cite{KPV2004} and references therein for more refined results on the variable coefficient {\Schr} equation.  The well-posedness is proved by the "gauged energy method" and the condition on the gauge captures the $a_2$ condition.  Ill-posedness is proved by an explicit geometrical optics construction.\\

The rest of the paper is organized as follows. In the section 2 the main results of the paper are stated. Well-posedness is proved in the section 3, and ill-posedness in section 4.\\

Some results of this paper were obtained during my Ph.D. studies at the University of Chicago, under the supervision of Carlos Kenig. I would like to thank Carlos Kenig and Cristian Rios for helpful discussions. Finally, I would like to thank the anonymous referee for helpful comments.
\section{Main results.}
 The following functional space notation is used. Let\\
  $\B_x^N = \{f(x)\in C^N(\R): \,\dx^j f \in L^\infty \text{ for all } 0\le i\le N\}$, $\B = \cap_n \B^N$, and\\ $H^s = \{ f\in \mathscr{S}': \norm{f}_{H^s} = \norm{\jap{\xi}^s \hat f(\xi)}_{L^2} <\infty\}$, where $\jap{x} = \sqrt{1+\abs{x}^2}$.\\
   For $1\le p<\infty$, define \[\norm{u}_{L^p_{[0,T]}X_x}:=(\int_0^T \norm{u(t)}^p_{X_x}dt) ^{\frac{1}{p}} \text{ and } X_T:=\norm{u}_{L^\infty_{[0,T]}X_x}:=\esssup_t \norm{u(t)}_{X_x}\]
for one of spaces $X$ above. \\

The following assumptions are made for the coefficients of \eqref{eq}
\begin{description}
  \item[(A1)] Dispersive coefficient $a_3(t,x)$ is non-degenerate. That is, there are constants $\Lambda\ge \lambda>0$, such that
  \begin{equation*}
  \lambda \le \abs{a_3(t,x)} \le \Lambda
\end{equation*}
uniformly for $(x,t) \in \R\times[0,T]$.
\item[(A2)] Regularity of the coefficients. For all $N\ge 0$.
\begin{itemize}
  \item $a_3 \in C^0_{[0,T]}\B_x^{N+3}\cap C^1_{[0,T]} \B^1_x$.
  \item $a_2 \in C^0_{[0,T]}\B_x^{N+2}\cap C^1_{[0,T]} \B^0_x$.
  \item $a_1 \in C^0_{[0,T]}\B_x^{N+1}$
  \item $a_0 \in C^0_{[0,T]}\B_x^N$.
\end{itemize}
\item[(A3)] Weak diffusion. $\int_0^x \frac{ \,a_2(y,t)}{\abs{a_3(y,t)}}d y \in C^1_{[0,T]}L^\infty_x$.
\end{description}
Note, that by $(A1)$ and $(A2)$, $a_3$ has a constant sign.\\

For $N\ge 0$ define
\begin{align*}
 \hspace{-20 pt} & C_N = \norm{a_3}_{L^\infty_T} + \norm{\frac{1}{a_3}}_{L^\infty_T} + \sum_{j=0}^3 \norm{a_j}_{\B^{N+i}_T} + \sum_{i=2}^3\norm{\dt a_j}_{L^\infty_T} +  \norm{\int_0^x \frac{a_2(y,t)}{\abs{a_3(y,t)}}d y}_{L^\infty_T}\\
 \hspace{-20 pt} & +\norm{\dt \int_0^x \frac{a_2(y,t)}{\abs{a_3(y,t)}}d y}_{L^\infty_T}
\end{align*}
For the well-posedness arguments, positive constants will depend on $C_N$ for some $N$ and will not be made explicit.
\begin{thm}
  \label{P} Suppose the coefficients of \eqref{eq} satisfy $(A1)$-$(A3)$. Then for all $s\in\R$, \eqref{eq} is well-posed in $H^s$. That is for any $(u_0,f) \in H^s \times L^1_{[0,T]}H^s$ there exists a unique $u\in C^0_{[0,T]}H^s$ satisfying \eqref{eq} in the sense of distributions. In addition, there exists $C=C(s)$
\begin{align}
   \label{Hs} \sup_{0\le t\le T}\norm{u(t)}_{H^s} \le C e^{CT}(\norm{u_0}_{H^s}+\int_0^T\norm{f(t)}_{H^s} dt)
\end{align}
Moreover, for any $\delta>\frac{1}{2}$, the solution additionally satisfies $u\in L^2_{[0,T]} H^{s+1}_{\jap{x}^{-2\delta} dx}$ and there is a $\tld C=\tld C(s,\delta)$
\begin{align}
  \label{Sm}    \norm{\jap{x}^{-\delta}\dx u}_{L^2_{[0,T]}H^s_x} \le \tld C(1+\sqrt{T})e^{\tld CT}(\norm{u_0}_{H^s} +\int_0^T\norm{f(t)}_{H^s}dt)
\end{align}
\end{thm}
Estimate \eqref{Hs} implies continuous dependence for \eqref{eq}, while estimate \eqref{Sm} is a manifestation of a \textit{local smoothing effect} of \eqref{eq}.
\begin{remark}
  If in addition, $f\in C^0_{[0,T]}H^{s-3}$, then for $s>3\frac{1}{2}$ the unique solution from the Theorem \ref{P} is classical by the Sobolev embedding.
\end{remark}
\begin{remark}\label{-t}
If the coefficients of \eqref{eq}, in addition, satisfy $(A1)$ - $(A3)$ on $[-T,0]$, then the transformation of the equation by $t\to -t$ changes the sign of all $a_j$, while again preserving all of the assumptions. Therefore, Theorem \ref{P} extends to $[-T,0]$.\\
Moreover, the transformation $x\to -x$ in \eqref{eq} changes the sign of $a_j$ for odd $i$, but preserves the assumptions $(A1)$ - $(A3)$. Without of loss of generality $a_3 >0$ will be assumed.
\end{remark}
Ill-posedness result  complements the Theorem \ref{P} and is proved by a different method.
\begin{thm}\label{Neg}
  Suppose the coefficients of \eqref{eq} satisfy $(A1)$, $(A2)$ 
  and
  \begin{description}
    \item[(A3N)] $\sup_{x>0}\int_0^x \frac{ \,a_2(y,0)}{\abs{a_3(y,0)}}d y=\infty$
  \end{description}
  Then for all $T>0$ and $s\in \R$ \eqref{eq} is ill-posed in $C^0_{[0,T]}H^s$ forward in time. More precisely, there is no continuous function $C(t,t_0)$ for $0\le t_0\le t\le T$, such that
  \begin{align}\label{ill}
    \sup_{t_0\le t\le T}\norm{u(t)}_{H^s} \le C(t,t_0) \norm{u(t_0)}_{H^s}
  \end{align}
  whenever $u$ solves \eqref{eq} on $[0,T]$ with $f\equiv 0$. Equivalently \eqref{Hs} fails on any $[0,T]$.
\end{thm}
 \begin{remark}\label{a3+}
   The transformation $x\to -x$ shows that $(A3N)$ is equivalent to
   \[\sup_{x<0}\int_x^0 \frac{ \,a_2(y,0)}{\abs{a_3(y,0)}}d y=\infty.\]
   However, the equivalence breaks down if absolute values are removed from $a_3$ in $(A3)$. Thus $a_3>0$ can be assumed without loss of generality, as long as $(A3N)$ is replaced with
 \begin{description}
     \item[(A3N')] $a_3>0$. Furthermore,  \begin{align*}
       \sup_{x>0}\int^x_0 \frac{ \,a_2(y,0)}{a_3(y,0)}d y=\infty \text{ or } \sup_{x<0}\int_x^0 \frac{ \,a_2(y,0)}{a_3(y,0)}d y=\infty
     \end{align*}
   \end{description}
   \end{remark}
\begin{remark}
  By reversing the time $t\to -t$ as in the Remark \ref{-t}, Theorem \ref{Neg} shows that
  \begin{align*}
    \sup_{x>0}\int_0^x \frac{ \,a_2(y,0)}{\abs{a_3(y,0)}}d y=-\infty
  \end{align*}
  leads to ill-posedness on $[-T,0]$. Thus the condition $\int_0^x \frac{ \,a_2(y,0)}{\abs{a_3(y,0)}}d y\in L^\infty$ is crucial for the well-posedness and the condition $(A3)$ for the Theorem \ref{P} is sharp for well-posedness on $[-T,T]$.
\end{remark}
While preparing this paper for publication, I have learned of a preprint by Ambrose-Wright \cite{AmWr12} that treats an analogue of \eqref{eq} in the periodic case. Their argument for the well-posedness is also based on the "gauged energy method", however in the case of $\R$ the smoothness of the coefficients does not imply integrability that is often needed. Additionally, this paper also proves that \eqref{eq} possesses a local smoothing effect, which is not present in the periodic case. The ill-posedness result in \cite{AmWr12} is done by a spectral method, which only works in the time independent case of \eqref{eq}.\\

\section{Well-posedness}
The main ingredient in the proof of the Theorem \ref{P} is stated as the following Proposition, which is an \textit{a priori} $L^2$ estimate for a slightly more general version of \eqref{eq}, that comes from commuting derivatives.
\begin{equation}
  \label{eq0}
  \begin{cases}
    \dt u + L_{A}u =f \text{ for } (t,x) \in (0,T]\times\R\\
    u(0,x)= u_0(x)
  \end{cases}\\ ,\text{ where } L_A= L + A_0(t,x,\dx)
\end{equation}
with $L$ from \eqref{eq}. The following assumptions are made on $A_0 \in C^0_{[0,T]}S^0$, the Pseudo-Differential operator of standard symbol class of order $0$ (Cf. Chapter VI of \cite{Stein93}):
\begin{description}
  \item[(A4)] The $S^0$ semi-norms of $A_0$ are bounded for $t\in[0,T]$ and their size depends on constants $C_N$ from $(A1)$--$(A3)$.
\end{description}
\begin{prop}\label{L2}
Suppose that the coefficients $a_j$ of \eqref{eq} satisfy $(A1)$--$(A3)$ and $A_0$ satisfies $(A4)$. Then there exists a constant $C$ and for any $\delta>\frac{1}{2}$ there is a constant $\tld C$, such that for any $u\in C^1_{[0,T]}L^2\cap C^0_{[0,T]}H^{3}$, the triple $(u,u_0,f)$ with $u_0$ and $f$ defined by \eqref{eq0} satisfies
\begin{align}
\label{L2e} \sup_{0\le t\le T}\norm{u(t)}_{L^2} \le C e^{CT}(\norm{u_0}_{L^2}+\int_0^T\norm{f(t)}_{L^2} dt)\\
  \label{L2g}    \norm{\jap{x}^{-\delta}\dx u}_{L^2_{[0,T]\times x}} \le \tld C(1+\sqrt{T})e^{\tld CT}(\norm{u_0}_{L^2} +\int_0^T\norm{f(t)}_{L^2}dt)
\end{align}
\end{prop}
\begin{remark}
  \label{rem:L2} If $A_0\equiv 0$, then $N=0$ in $(A2)$ can be chosen for the Proposition \ref{L2}.
\end{remark}
The proof of the Proposition \ref{L2} is done by a change of variables (gauge) followed by the application of the energy estimates. The proof is broken into several preliminary results.\\

 A gauge is a smooth invertible function, which for the purposes of the argument needs to have $3$ bounded derivatives:
\begin{defn}\label{gau}
  A function $\phi\in C^0_{[0,T]}\B^3_x\cap C^1_{[0,T]}\B^0$ is called a \textbf{gauge}, if
  \begin{itemize}
    \item $\phi(x,t)>0$ with $\frac{1}{\phi}\in L^\infty_{[0,T]\times\R}$.
    \item  $\norm{\frac{1}{\phi}}_{L^\infty_{[0,T]\times\R}}+\norm{\phi}_{\B^3_T}+ \norm{\dt \phi}_{L^\infty_T} \le C(C_0,\delta)$ with $C_N$ from $(A1)$--$(A3)$.
      \end{itemize}
\end{defn}
Suppose that $\phi(x,t)$ is a \textit{gauge}. Define
\begin{align*}
 v= \phi^{-1} u
\end{align*}
Definition \ref{gau} implies that $v \in C^1_{[0,T]}L^2\cap C^0_{[0,T]}H^{3}$ if and only if $u \in C^1_{[0,T]}L^2\cap C^0_{[0,T]}H^{3}$ and substitution of $v$ into \eqref{eq0} gives:
\begin{align}\label{vt}
  \begin{cases}
    \dt v + L_\phi v = \phi^{-1}f \\
    v(x,0)= \phi^{-1}u_0
  \end{cases}
\end{align}
where
\begin{align*}
L_\phi = a_3\dx^3 +\left(a_2 + \phi^{-1}{3 a_3\dx \phi}\right)\dx^2  +\left(a_1+ \phi^{-1}({2 a_2\dx \phi+ 3 a_3\dx^2\phi})\right)\dx\\
  + \left(a_0 +{\phi^{-1}}({\dt \phi+a_1\dx \phi+a_2\dx^2\phi+a_3\dx^3 \phi})\right)I + \phi^{-1}A_0(\phi  \underline{\,\,\,})
\end{align*}

\begin{lem}\label{v}
From the definition of the gauge,
 \begin{align}\label{phi}
   \norm{u}_{L^2} \approx \norm{v}_{L^2} \text{ and } \sum_{j=0}^1\norm{\jap{x}^{-\delta}\dx^i u}_{L^2_{[0,T]\times x}} \approx \sum_{j=0}^1\norm{\jap{x}^{-\delta}\dx^j v}_{L^2_{[0,T]\times x}}
 \end{align}
 with comparability constants dependent only on the constant in the Definition \ref{gau}. Therefore, to prove Proposition \ref{L2} it suffices to prove \eqref{L2e} and \eqref{L2g} for $v$ satisfying \eqref{vt}.
\end{lem}
\begin{proof}
 It suffices to show one sided inequalities in \eqref{phi} as $\phi^{-1}$ satisfies the same estimates as $\phi$. The first comparability follows from $\norm{u}_{L^2}\le \norm{\phi}_{L^\infty}\norm{v}_{L^2}$. For the second, a similar computation and Cauchy-Schwartz implies
 \begin{align*}
  (\sum_{j=0}^1\norm{\jap{x}^{-\delta}\dx^i u}_{L^2_{[0,T]\times x}})^2 \le 2( \norm{\phi}_{L^\infty}^2+ \norm{\dx\phi}_{L^\infty}^2) (\sum_{j=0}^1\norm{\jap{x}^{-\delta}\dx^i v}_{L^2_{[0,T]\times x}})^2
 \end{align*}
 It is clear from \eqref{phi} that \eqref{L2e} is equivalent for $u$ and $v$, whereas an estimate
 \begin{align*}
   \norm{\jap{x}^{-\delta}\dx u}_{L^2_{[0,T]\times x}}\le C\sum_{j=0}^1\norm{\jap{x}^{-\delta}\dx^i v}_{L^2_{[0,T]\times x}}\le C(\sqrt{T}\norm{v}_{L^2_T}+\norm{\jap{x}^{-\delta}\dx v}_{L^2_{[0,T]\times x}})
 \end{align*}
 implies \eqref{L2g} for $u$, if \eqref{L2e} and \eqref{L2g} hold for $v$.
\end{proof}


The energy method involves multiplying \eqref{vt} by $v$ to estimate $\dt \norm{v}_{L^2}^2$ by $\norm{v}_{L^2}^2$:
\begin{align*}
  \dt \int \abs{v}^2 = - 2 Re( L_\phi v,v) + (f,\phi v)
\end{align*}
The following Lemma summarizes the energy estimates for $L$ or $L_\phi$:
\begin{lem}\label{enrg}
  Consider an operator $L=a_3\dx^3 + a_2\dx^2 + a_1\dx +a_0$, where $a_3$--$a_0$ satisfy $(A2)$. Then for $v\in C^0_{[0,T]}H^{3}$
  \begin{align*}
    Re (Lv,v) = (\left[-a_2+\frac{3}{2} \dx a_3\right]\dx v, \dx v) + (b_0 v,v)
  \end{align*}
 for $b_0=a_0-\frac{1}{2}(\dx a_1-\dx^2 a_2 +\dx^3 a_3)$, where $(u,v)$ is an $L^2_x$ pairing.
\end{lem}
\begin{proof}[Proof of Lemma \ref{enrg}]
The computation is immediate by computing the adjoint $L^*$ of $ L$ using the Calculus of PDO. Alternatively, as $ L$ is a differential operator, the same computation can be also done by a repeated integration by parts. Indeed, the operator $\dx^{k}$ is skew-adjoint for odd $k$, which implies that principal parts of odd order terms are eliminated by integration by parts. For example
  \begin{align*}
    (a_1 \dx v, v) = - (v , a_1\,\dx v) -(\dx a_1 \,v,v) = -\overline{(a_1\, \dx v, v)} -(\dx a_1\, v,v)
  \end{align*}
An identical computation shows
\begin{align*}
Re (a_3\dx^2 v, \dx v) = - \frac{1}{2}(\dx a_3 \dx v, \dx v)\text{ and }Re (\dx^2 a_3 \dx v, v) = -\frac{1}{2} (\dx^3\, a_3 v,v)
\end{align*}
 Using these identities and more integration by parts establishes
  \begin{align*}
    Re(a_3\dx^3 v,v) = \frac{3}{2}(\dx a_3 \,\dx v,\dx v)  -\frac{1}{2} (\dx^3\, a_3 v,v)
  \end{align*}
A similar analysis for $Re (a_2\,\dx^2 v,v)$ completes the proof.
\end{proof}
Applying Lemma \ref{enrg} to $L_\phi$, shows that the only term of order higher than $0$ is\\ $  \left([2 a_2 + \frac{6 a_3\dx \phi}{\phi}- 3 \dx a_3]\dx v, \dx v\right)$. Thus, if this term were negative, an a priori estimate would be obtained for $v$. This motivates the choice of a gauge $\phi$ that should satisfy
\begin{align*}
  2 a_2 + \phi^{-1}6 a_3\dx \phi- 3 \dx a_3\le 0
\end{align*}
A choice of equality in this equation can be made and this choice is enough for the estimate \eqref{L2e}, but by exploiting the inequality the local smoothing estimate \eqref{L2g} is proved. The exact choice of a gauge is summarized in the following Lemma
\begin{lem}
  \label{gauCh}
  For $\delta >\frac{1}{2}$, let $\phi(x,t)$ be a solution of the ODE
  \begin{align*}
    \begin{cases}
      6 a_3 \dx \phi = \left( 3 \dx a_3 -c_{\delta}\jap{x}^{-2\delta} - 2a_2\right) \phi\\
      \phi(t,0) = 1
    \end{cases}
  \end{align*}
  where $c_\delta=0$ or $1$. Then $\phi$ is a gauge in the sense of the Definition \ref{gau}, and is  independent of $\delta$ if $c_\delta=0$.
\end{lem}
\begin{proof}
  The ODE for $\phi$ is solved explicitly as
  \begin{align*}
    \phi(x,t) = \sqrt{\frac{a_3(x,t)}{a_3(t,0)}}e^{-\int_0^x \frac{a_2(y,t)}{3a_3(y,t)}dy}e^{-\int_0^x\frac{c_\delta dy}{6a_3(y,t)\jap{y}^{2\delta}}}
  \end{align*}
  By $(A3)$ $e^{-\int_0^x \frac{a_2(y,t)}{3a_3(y,t)}dy} \approx 1$. $(A1)$ implies $\sqrt{\frac{a_3(x,t)}{a_3(t,0)}}\approx 1$.\\Finally, as $\delta >\frac{1}{2}$,
\begin{align*}
e^{-\int_0^x\frac{c_\delta dy}{6a_3(y,t)\jap{y}^{2\delta}}} =\begin{cases}
  1, & \text{ if } c_\delta =0\\
  C(\delta), & \text{ if } c_\delta =1
\end{cases}
\end{align*}
A computation for $\dt \phi$ and $\dx^j \phi$ for $j=1$, $2$ and $3$ and using $(A1)$--$(A3)$ finishes the proof.
\end{proof}
 \begin{proof}[Proof of Proposition \ref{L2}]
 By the Lemma \ref{v} it suffices to prove the Proposition for $v$ satisfying \eqref{vt}. \\

 Applying the Lemma \ref{enrg} for $L_\phi$ implies that
\begin{align*}
 & \dt \int \abs{v}^2 dx =  (\left[2 a_2 + \frac{6 a_3\dx \phi}{\phi}- 3 \dx a_3\right]\dx v, \dx v) + (\tld b_0 v,v) - 2 Re( A_0(\phi v),\phi v) + (f,\phi v)
\end{align*}
where $\tld b_0$ is obtained from the Lemma \ref{enrg} applied to $L_\phi$. With $\phi$ chosen from the Lemma \ref{gauCh}, this implies
\begin{align*}
   \dt \int \abs{v}^2 dx \le  -c_\delta(\jap{x}^{-2\delta}v,v) + (\tld b_0 v,v) - 2 Re( A_0(\phi v),\phi v) + (f,\phi v)
\end{align*}
 By (A4), $A_0:L^2\to L^2$ is bounded. Moreover, by the Definition \ref{gau} and $(A2)$, $\phi \in L^\infty$ and $\tld b_0\in L^\infty$. Hence
\begin{align*}
  \dt \int \abs{v}^2 \le C(\int \abs{v}^2dx + \norm{v}_{L^2}\norm{f}_{L^2}) -\norm{\jap{x}^{-\delta} \dx v}_{L^2}^2
\end{align*}
For $c_\delta=0$ an application of Grownwall Lemma implies \eqref{L2e} for $v$.\\

Whereas moving $\dx v$ term to the left hand side for $c_\delta=1$ and integrating in time gives
\begin{align*}
  & \int_0^T \norm{\jap{x}^{-\delta}\dx v}^2 dt \le C\int_0^T(\int \abs{v}^2dx + \norm{v}_{L^2}\norm{f}_{L^2})dt +\norm{v_0}_{L^2}^2 -\norm{v}_{L^2}^2\\
  & \le (C(1+T)-1)\sup_{0\le t\le T}\norm{v(t)}_{L^2}^2 +\norm{v_0}_{L^2}^2 + (\int_0^T\norm{f(t)}_{L^2} dt)^2
\end{align*}
Using \eqref{L2e} completes the proof of \eqref{L2g}.
\end{proof}
Proposition \ref{L2} can be strengthened to an $H^s$ estimate.
\begin{prop}  \label{apri} Let $L$ be as in \eqref{eq}, whose coefficients $a_j$ satisfy $(A1)$--$(A3)$. Then for any $s\in \R$ there exist  constants $C(s)$ and $\tld C(s,\delta)$ for any $\delta >\frac{1}{2}$, such that for any $u\in C^1_{[0,T]}H^s\cap C^0_{[0,T]}H^{s+3}$ the following estimates hold
    \begin{align}\label{apr1}
      \sup_{0\le t\le T} \norm{u(t)}_{H^s_x} \le Ce^{CT}(\norm{u(0)}_{H^s_x} + \int_0^T\norm{\dt u + Lu}_{H^s_x}dt)
    \end{align}
    \vspace{-10 pt}
    \begin{align*}
     & \sup_{0\le t\le T} \norm{u(t)}_{H^s_x} \le Ce^{CT}(\norm{u(T)}_{H^s_x} + \int_0^T\norm{-\dt u + L^*u}_{H^s_x}dt)
     \end{align*}
     where $L^*$ is the adjoint of $L$. Moreover
     \begin{align*}
     & \norm{\jap{x}^{-\delta}\dx u}_{L^2_{[0,T]}H^s_x} \le \tld C(1+\sqrt{T})e^{\tld CT}(\norm{u_0}_{H^s} +\int_0^T\norm{f(t)}_{H^s}dt)
    \end{align*}
\end{prop}
\begin{coro}\label{Hor}
  By the Theorem 23.1.2 on page 387 in \cite{Hor:v3}, the proof of the Theorem \ref{P} reduces to the Proportion \ref{apri}.
\end{coro}
The Proposition \ref{apri} is reduced to the Proposition \ref{L2}. Observe, that
\begin{align*}
f= \dt u + Lu \text{ if and only if } J^s f = \dt J^s u + L J^s u + [J^s L] J^{-s} J^s u
\end{align*}
where $J^s$ is a Pseudo Differential Operator with symbol $\jap{\xi}^{s}$. Therefore to prove \eqref{apr1} it suffices to show that the Proposition \ref{L2} applies to the operator $\tld L = L +[J^s L] J^{-s}$.
\begin{lem}\label{Ls}
  Let $\tld L = L +[J^s L] J^{-s}$ with $L$ from \eqref{eq} that satisfies $(A1)$ and $(A2)$. Then
  \begin{align}\label{Lsf}
  \begin{split}
    & \tld L =a_3\dx^3 + a_0  +\sum_{i=1}^2 (a_j + \tld a_j) \dx^j + A_s(t, x,\dx)\\
    & \text{ with } \tld a_2 = s \dx a_3 \text{ and } \tld a_1 = s\dx a_2 + \frac{s(s-1)}{2} \dx^2 a_3
  \end{split}
  \end{align}
  where $A_s \in S^0$, whose semi-norms depend on the coefficient bounds $(A2)$ for $N=N(s)$ and hence satisfies $(A4)$.\\

  Furthermore, the coefficients $\tld a_j$ for $i=1$, $2$ satisfy $(A2)$--$(A3)$.
\end{lem}
\begin{proof}
From the first term in the Calculus of PDO $[J^s L] J^{-s}\in S^{2}$. A further expansion of $[J^s, a_3\dx^3]$ gives:
\begin{align*}
  \sigma([J^s, a_3\dx^3])=\sum_{1\le \abs{\alpha}\le 2} \frac{i^{-\abs{\alpha}}}{\alpha!} \dxi^\alpha \jap{\xi}^s \dx^\alpha (a_3(i\xi)^3) \mod S^{s}\\
  = s\dx a_3 (i\xi)^2\jap{\xi}^s + \frac{s(s-1)}{2} \dx^2 a_3 i\xi \jap{\xi}^s \mod S^{s}
\end{align*}
where the substitution $\xi^2 = \jap{\xi}^2 - 1$ was used and the terms of order $s$ were absorbed into the remainder. Performing a similar computation for the remaining terms in $[J^s L]$ and composition with $J^{-s}$ verifies \eqref{Lsf}.\\

It is immediate from \eqref{Lsf} that $\tld a_j$ satisfies $(A2)$. To verify $(A3)$ observe that
\begin{align*}
  \int_0^x \frac{ \,\tld a_2(y,t)}{\abs{a_3(y,t)}}d y  = s \sign(a_3) \log \frac{a_3(x,t)}{a_3(0,t)}\in C^1_{[0,T]}L^\infty_x
\end{align*}
by $(A1)$ and $(A2)$.
\end{proof}
\begin{remark}\label{L*}
  A simple computation shows that the adjoint $L^*$ of the operator $L$ from \eqref{eq} is
\begin{align*}
  L^* = &-a_3\dx^3 +  (a_2-3 \dx a_3)  \dx^2 + (a_1+2\dx a_2 - 3 \dx^2 a_3)\dx\\
  & +(a_0-\dx a_1+\dx^2 a_2 -\dx^3 a_3)
\end{align*}
whereas a substitution $t\to T-t$ transforms \eqref{eq} to
\begin{align*}
\begin{cases}
  -\dt u(T-t) + L u (T-t)= f(T-t)\\
  u(T-t)\mid_{t=0} = u(T)
\end{cases}
\end{align*}
Both $L^*$ and $L(T-t)$ satisfy $(A1)$--$(A3)$.
\end{remark}
\begin{coro}
  Lemma \ref{Ls}, Remark \ref{L*} and the Proposition \ref{L2} imply the Proposition \ref{apri}.
\end{coro}
This completes the proof of Theorem \ref{P} by the Corollary \ref{Hor}.
\section{Ill-posedness}\label{ills}
Ill-posedness is proved by justifying the formal geometrical optics argument, cf \cite{CrGdm90}, for a special choice of initial data. It is instructive to first consider the case of constant dispersion $a_3\equiv 1$:
\begin{align}\label{cd}
  \dt v + \dx^3 v + \sum_{j=0}^2 c_2(x,t) \dx^j v= g
\end{align}
Then the condition $(A3N')$ is equivalent to \[\sup_{N>0,x_0}\int_{x_0-N}^{x_0} c_2(x',0) dx'=\infty\] The general case of \eqref{eq} is later reduced to illposedness for \eqref{cd}. For this reduction it is desirable to relax the condition (A2) to smooth, but not necessarily bounded coefficients:
\begin{description}
  \item[(A2')] Assume that $c_2 \in C^1_t C^0_x\cap C^0_tC^2_x$ and $c_1$, $c_0 \in C^0_{t,x}$.
\end{description}
  From now on, the notation $C=C(\alpha)$ means that there exists a constant $C\ge 1$ that depends continuously $\alpha$ and may depend on the norms of coefficients $c_j$ evaluated on some compact set, whose size also depends on $\alpha$. Constants required to be small are reciprocal of the large constants.\\

 The proof of the illposedness for \eqref{cd} rests on the following explicit construction that violates the estimate \eqref{Hs} for $s=0$. Let $\psi(x) =\eta^{-\frac{1}{2}}\psi_0(\frac{x-x_0}{\eta})$, where $\psi_0\in C^\infty_0([-1,1])$, $\norm{\psi_0}_{L^2}=1$ and the small parameter $0<\eta\le 1$ is to be chosen. Then
 \begin{align}\label{psi}
   \supp \psi \subset [x_0-\eta,x_0+\eta],\,\, \norm{\psi}_{L^2}=1\text{ and }\norm{\psi}_{H^k} \le C\eta^{-k}\text{ for }k\ge 0.
 \end{align}
Define
\begin{align}\label{cdw}
  v(x,t) :=e^{iS} w, \text{ with }S=x\xi + t\xi^3\text{ and }w=e^{\frac{1}{3}\int_x^{x_0} c_2(x',t) dx'} \psi(x+3\xi^2 t).
\end{align}
with parameters $\xi\ge 1$, $x_0$, $0<\eta\le 1$ to be chosen. It is immediate from (A2') that $w\in C^1_t C^0_x \cap C^0_tC^3_x$. A substitution of the ansatz $v=e^{iS} w$ into \eqref{cd} gives
\begin{align*}\hspace{-30 pt}\begin{split}
 & g= \frac{1}{3}\int_x^{x_0} \dt c_2(x',t) dx'\cdot v +e^{iS}\left\{( 3 i \xi \dx^2 w + \dx^3 w) + 2 c_2  i \xi \dx w +c_2 \dx^2 w+c_1(i \xi \cdot w + \dx w) + c_0 w \right\}.
\end{split}
\end{align*}
Taking absolute values gives
\begin{align}\label{rem}
  \abs{g(x,t)}\le \xi \sum_{j=0}^3g_j(x,t)\abs{\dx^j w(x,t)}
\end{align}
 where $g_j(x,t)$ are continuous non-negative functions independent of $\xi$.\\

Observe from \eqref{psi} and \eqref{cdw}, that $\supp_x w(x,t) \subset [x_0-3\xi^2t-\eta ,x_0-3\xi^2t+\eta]$. Therefore,
 \begin{align*}
   \int_x^{x_0} c_2(x',t) dx'= \int_{x_0-3\xi^2t}^{x_0} c_2(x',t) dx' +I(x,t)
 \end{align*}
 where on the support of $w(x,t)$, $\abs{I(x,t) } \le C(x_0-3\xi^2t)\eta$.\\

 Using \eqref{cdw}, \eqref{psi} and the estimate above implies $0<\eta \le \tfrac{1}{C(x_0-3\xi^2t)}$ gives
 \begin{align}\label{cd-s}
   \norm{v_n(t)}_{L^2_x} \approx_2 e^{\frac{1}{3}\int_{x_0-3\xi^2t}^{x_0} c_2(x',t) dx'}
 \end{align}
 with comparability constant chosen to be $2$. The estimates \eqref{rem} and \eqref{cd-s} are the main ingredient for the proof of the following theorem.
 \begin{thm}\label{illc}
   Suppose
   \begin{align}\label{a2c}
     \sup_{N>0,x_0}\int_{x_0-N}^{x_0} c_2(x',0) dx'=\infty
   \end{align}
    Then there exists a sequence $t_n\to 0$, and sequences $x_0^n$ and $\xi_n$, $\eta_n$ such that $v_{n} \in C^1_tL^2_x\cap C_t^0 H^{3}_x$ from \eqref{cdw} and $g_n$ from \eqref{cd} satisfy
  \begin{align}\label{ille}
\norm{v_n(t_n)}_{L^2_x} \ge n(\norm{v_n(0)}_{L^2_x}+\int_0^{t_n}\norm{g_n(t)}_{L^2_x} dt)>0\\
\label{gos}\int_0^{t_n}\norm{v_n(t)}_{L^2_x} \le \frac{1}{n}\norm{v_n(0)}_{L^2}
  \end{align}
 \end{thm}
 \begin{proof}
   By \eqref{a2c}, there exist $x_0\in \R$ and $N>0$, such that
   \begin{align}\label{A3N}
     e^{\frac{1}{3}\int_{x_0-N}^{x_0} c_2(x',0) dx'} \ge 16 n
   \end{align}
   Let $t_n =\dfrac{N}{3\xi^2}$ with $\xi=\xi(x_0,N)$ to be chosen below. From now on only $t$, such that $0\le t\le t_n$ will be considered. For this range of $t$, the small parameter $\eta=\eta(x_0-3\xi^2t)>0$ can be chosen to depend only on $(x_0,N)$. As the choice of $x_0$ and $\eta$ completely determines $\psi$, $\psi$ is independent of $\xi$.\\

To estimate the right hand side of \eqref{ille}, observe from \eqref{cd-s},
   \begin{align*}
      \frac{1}{2}\le \norm{v_n(0)}_{L^2_x} \le 2
   \end{align*}
    Furthermore, \eqref{psi} and \eqref{cdw} imply that $\supp_x w(x,t) \subset [x_0-N-1,x_0+1]$ for $0\le t\le t_n$. Hence, $w$, $v$ and $g$ have compact supports independent of $\xi$ and are bounded. More precisely, \eqref{psi}, \eqref{cdw} and \eqref{rem} imply \[ \norm{g(t)}_{L^2} \le C(N,x_0)\xi \] Integrating this inequality in time gives
   \begin{align}
    \int_0^{t_n}  \norm{g(t)}_{L^2} dt \le \frac{C(N,x_0)}{\xi}
   \end{align}
   Therefore, for $\xi\ge C(N,x_0)$, $\int_0^{t_n}  \norm{g(t)}_{L^2} \le 1$. This finishes the analysis of the right hand side of \eqref{ille}.\\
    Similarly, \eqref{cd-s} implies that for $0\le t\le t_n$, $      \norm{v(t)}_{L^2_x} \le C(x_0,N)$. Hence, for $\xi\ge C(x_0,N) $
      \[\int_0^{t_n}\norm{v_n(t)}_{L^2} dt \le \frac{1}{2n}\le\frac{1}{n}\norm{v_n(0)}_{L^2_x}\]

  To finish the proof it suffices to show that there exists $\xi=\xi_n(x_0,N)\ge C(x_0,N)$, such that
   \begin{align}\label{c-s}
     \norm{v_n(t_n)}_{L^2_x} \ge 4n.
   \end{align}
   This estimate requires a comparison of \eqref{cd-s} and \eqref{A3N}. To this end, by \eqref{cd-s} and the Fundamental Theorem of Calculus
   \begin{align*}
    \norm{v_n(t_n)}_{L^2_x} \ge \frac{1}{2} e^{\frac{1}{3}\int_{x_0-3\xi^2 t_n}^{x_0} c_2(x',0) dx'} e^{\frac{1}{3}\int_{x_0-3\xi^2 t_n}^{x_0} \int_0^{t_n}\dt c_2(x',t) dx'dt}
   \end{align*}
   Whereas, using $t_n=\tfrac{N}{3\xi^2}$,
   \begin{align*}
   \abs{\int_{x_0-N}^{x_0} \int_0^{\frac{N}{3\xi^2}} \dt c_2(x',t)dt dx'} \le \frac{C(x_0,N)}{\xi^2} \le \log 2
 \end{align*}
for $\xi\ge C(x_0,N)+1$. Combining the last two estimates and \eqref{A3N} implies \eqref{c-s}.
 \end{proof}
 \subsection{Reduction to constant dispersion}
 Illposedness for \eqref{eq} relies on a change of variables to reduce to \eqref{cd}.
 \begin{defn}\label{y}
  For $ a_3$ satisfying (A1) and (A2) define
 \begin{align}\label{yeq}
   y(x,t) = \int_0^x a_3^{-\frac{1}{3}}(x',t)dx'
 \end{align}
 \end{defn}
 This construction allows to replace the roles of $x$ and $y$ as follows.
 \begin{lem}\label{x}
 Consider
   \begin{align}\label{y2}
   y-y(x,t) = 0
 \end{align}
 with $y(x,t)$ from \eqref{yeq}. Then there exists a unique smooth function $x=x(y,t)$ that satisfies \eqref{y2}. Moreover,
 \begin{align*}
   \der{x}{y} = \frac{1}{\der{y}{x}} = a_3^{\frac{1}{3}}(x,t)
 \end{align*}
 \end{lem}
 \begin{proof}
   By (A1) and Fundamental Theorem of Calculus $\der{y}{x}(x,t)=a_3^{-\frac{1}{3}}(x,t)\neq 0$ for all $(x,t)$. An application of the Implicit Function Theorem for \eqref{y2} completes the proof.
 \end{proof}
Define
 \begin{align*}
v(y,t) = a_3^{-\frac{1}{3}}(x(y,t),t) u(x(y,t),t)
 \end{align*}
 using the Lemma \ref{x}. Equivalently
 \begin{align}\label{illg}
   u(x,t) = \frac{1}{\der{y}{x}} v(y,t).
 \end{align}
 From this definition, $L^2$ norms of $u$ and $v$ are comparable by (A1):
 \begin{align}\label{cm}
   \norm{u(t)}_{L^2_x}^2 = \int a_3(x,t) \abs{v(y,t)}^2 dy \approx_{\lambda,\Lambda} \norm{v(t)}_{L^2_y}^2.
 \end{align}
 A computation shows, that
 \begin{align*}
 & \dt u = \frac{1}{\der{y}{x}}(\dt v + \der{y}{t}\dy v - \frac{\dt\der{y}{x}}{\der{y}{x}}v)  & \dx u = \dy v - v\frac{\dx^2 y}{(\der{y}{x})^2}\\
& \dx^2 u = \dy^2 v \der{y}{x} +  \sum_{j=0}^1 b_j(\dx y,\dx^2 y)\dy^j v  & \dx^3 u = \dy^3 v (\der{y}{x})^2 + \sum_{j=0}^1 \tld b_j(\dx y,\dx^2 y,\dx^3 y)\dy^j v
 \end{align*}
 for smooth functions $b_j$ and $\tld b_j$. Using this computation and $a_3(x,t)(\der{y}{x})^3\equiv 1$ substitute $\eqref{illg}$ into \eqref{eq} to get
 \begin{align}\label{cdc}
   \dt v + \dy^3 v + \sum_{j=0}^2 c_j(y,t) \dy^j v = g
 \end{align}
where the coefficients $c_j$ satisfy \textbf{(A2')} and, in particular
\begin{align}\label{c2}
c_2(y,t) = a_2(x,t) a_3^{-\frac{2}{3}}(x,t);\hspace{10 pt} & g(y,t) = \der{y}{x} f(x,t)
\end{align}
The relationship between $f$ and $g$ is identical to \eqref{illg}, thus \eqref{cm} implies
\begin{align}\label{fg}
  \norm{f(t)}_{L^2_x}\approx_{\lambda,\Lambda} \norm{g(t)}_{L^2_y}
\end{align}
Therefore, \eqref{eq} can be reduced to \eqref{cd}, which was analyzed in the Theorem \ref{illc}.
\begin{lem}\label{Negc}
  Suppose $(A1)$, $(A2)$ and $(A3N')$ hold. Let $s\in \R$. Then there exists a sequence $u_n\in C^1H^s\cap C_t^0 H^{s+3}$ and $t_n\to 0$, such that
  \begin{align}\label{N}
    \norm{w_n(t_n)}_{H^s} \ge n(\norm{w_n(0)}_{H^s_x}+\int_0^{t_n}\norm{(\dt + L)w_n(t)}_{H^s_x} dt)>0
  \end{align}
\end{lem}
Note, that \eqref{a2c} for $c_2(y,t)$ defined by \eqref{c2} is equivalent to \textbf{(A3N')}. Therefore, Theorem \ref{illc} applies to \eqref{cdc}. Define $u_n$ by applying \eqref{illg} to $v_n$ from the Theorem \eqref{illc}, which can be written explicitly as
\begin{align}\label{un}
  u_n(x,t) = a_3^{\frac{1}{3}}(x,t)e^{iy\xi_n+it\xi_n^3} e^{\frac{1}{3}\int_x^{x_0} \frac{a_2}{a_3}(x',t)dx'}\psi_n( y+ 3\xi_n^2 t)
\end{align}
 Let $f_n=\dt u_n+ Lu_n$ and $g_n$ defined by \eqref{c2}. Then \eqref{ille}, \eqref{cm} and \eqref{fg} imply up to a subsequence
 \begin{align}\label{illu}
   \norm{u_n(t_n)}_{L^2_x} \ge n(\norm{u_n(0)}_{L^2_x}+\int_0^{t_n}\norm{(\dt + L)u_n(t)}_{L^2_x} dt)>0
 \end{align}
 Likewise, \eqref{gos} holds for $u_n$ instead of $v_n$. This completes the proof of \eqref{N} for $s=0$ by taking $w_n:=u_n$.\\

For the general $s\in \R$, commute $J^s$ with $L$ as in Lemma \ref{Ls}: $J^s(\dt + L) = (\dt + \tld L)J^s$, where $\tld L = L + [J^{s} L]J^{-s}$. By Lemma \ref{Ls} $\tld L = P +A_s(x,t,\dx)$, where $A_s\in S^0$ and the differential operator $P$ satisfies (A1), (A2) and (A3N'). Define $\tld u_n$ via \eqref{un} with $L$ replaced by $P$. I.e. $\tld u_n$ differs from $u_n$ by a factor of $(\frac{a_3(x_0,t)}{a_3(x,t)})^s$. Further define
\begin{align*}
  w_n(x,t)= J^{-s} \tld u_n(x,t)
\end{align*}
Hence $\norm{w_n(t)}_{H^s}=\norm{u_n(t)}_{L^2}$. Applying \eqref{illu} to $\tld u_n$ implies
\begin{align*}
      \norm{w_n(t_n)}_{H^s} \ge n(\norm{w_n(0)}_{H^s_x}+\int_0^{t_n}\norm{(\dt + P)\tld u_n(t)}_{L^2_x} dt)>0
\end{align*}
By \eqref{gos} for $n\ge \norm{A_s}_{L^2\to L^2}$
  \begin{align*}
     \int_0^{t_n} \norm{A_s \tld u_n}_{L^2} \le \norm{w_n(0)}_{H^s_x}
  \end{align*}
  As $J^s f = (\dt +P)\tld u_n + A_s\tld u_n$, combining the last two estimates implies that $w_n$ satisfies \eqref{N} by passing to a subsequence.
\subsection{Proof of illposedness}\label{reds}
\begin{coro}\label{illn}
  Lemma \ref{Negc} implies that \eqref{Hs} fails or, more generally, for any $T>0$ there is no non-decreasing function $C(T'):[0,T]\to\R$, such that
  \begin{align}\label{well}
    \sup_{0\le t\le T'}\norm{u(t)}_{H^s} \le C(T')(\norm{u_0}_{H^s}+\int_0^{T'}\norm{f(t)}_{H^s} dt)
  \end{align}
  holds for all $u\in C^0_{[0,T]}H^s$ solving \eqref{eq}.
\end{coro}
\begin{proof}
  Assuming \eqref{well}, for the sake of contradiction, and using \eqref{N}, implies that $C(t_n)\ge n$ for all $n\in\N$. As $t_n\to 0$ and $C(t)$ is non-decreasing in $t$, $C(t)\ge n$ for all $t>0$ and $n\in \N$. This is a contradiction.
\end{proof}
\begin{coro}
  Assuming (A1), (A2) and (A3N) implies, that \eqref{eq} is ill-posed in $H^s$ the sense of Theorem \ref{Neg}.
\end{coro}
\begin{proof}
  Suppose, for the sake of contradiction, that \eqref{ill} holds for some $[0,T]$ and some continuous function $C(t_0,t)$ for $0\le t_0\le t\le T$. Define a non-decreasing function $C(T')=\sup_{0\le t_0 \le t \le T'} C(t_0,t)$. Then by the Duhamel principle every solution of \eqref{eq} satisfies
  \begin{align*}
    u(t) = S(t,0)u_0 + \int_0^t S(t,t_0) f(t_0) dt_0
  \end{align*}
  where $u(t)=S(t,t_0)g$ solves \eqref{eq} on $[t_0,T]$ with data $u(t_0)=g$ and $f\equiv 0$. Moreover,
  \begin{align*}
    \sup_{0\le t_0 \le t \le T'}\norm{S(t,t_0)} \le C(T')
  \end{align*}
  Thus the Duhamel principle implies \eqref{well} for all $u\in C^0_{[0,T]}H^s$ solutions of \eqref{eq}, which contradicts the Corollary \ref{illn}.
\end{proof}


\bibliographystyle{amsplain}
\providecommand{\bysame}{\leavevmode\hbox to3em{\hrulefill}\thinspace}
\providecommand{\MR}{\relax\ifhmode\unskip\space\fi MR }
\providecommand{\MRhref}[2]{%
  \href{http://www.ams.org/mathscinet-getitem?mr=#1}{#2}
}
\providecommand{\href}[2]{#2}


\end{document}